\documentclass[12pt]{article}
\usepackage[body={6.5in,9.0in},left=1in,top=1in,centering]{geometry}
\usepackage{amsmath,bm,graphicx,amssymb,mathrsfs,amsthm,pstricks}
\usepackage[colorlinks=true,urlcolor=blue,linkcolor=blue,citecolor=blue,pdfstartview=FitH]{hyperref}

\usepackage{authblk}
\usepackage{graphicx}
\usepackage{amssymb}
\usepackage{epstopdf}
 \newcommand{\R}{\mathbb R}
\newcommand{\E}{\mathbb E}
\newcommand{\W}{\mathcal W}
 \newcommand{\cP}{\mathcal P}
\def\P{\mathbb P}

\newcommand{\tr}{\mathrm{tr}}

\newcommand{\mbf}{\mathbf}
\newcommand{\bmu}{\bm\mu} \newcommand{\bnu}{\bm\nu} 
  \usepackage[sort,comma]{natbib}
 \newcommand{\comment}[1]{}

 \allowdisplaybreaks
\numberwithin{equation}{section}
\newtheorem{theorem}{Theorem}[section]

\newtheorem{lemma}[theorem]{Lemma}
\newtheorem{assumption}[theorem]{Assumption}
\newtheorem{corollary}[theorem]{Corollary}
\theoremstyle{definition}
\newtheorem{example}{Example}[section]
\newtheorem{definition}[theorem]{Definition}
\newtheorem{rem}[theorem]{Remark}

\title{On Lyapunov Conditions for the Well-Posedness of McKean-Vlasov Stochastic Differential Delay Equations}
\author[1]{Dan Noelck}
\affil[1]{Department of Applied Mathematics, Illinois Institute of Technology. {\tt dnoelck@iit.edu},
}
\date{\today}

\begin{document}

\maketitle
\begin{abstract}
This work focuses on the well-posedness of McKean-Vlasov stochastic differential delay equations.  Under suitable Lipschitz conditions on the drift and diffusion terms, along with a distribution dependent Lyapunov condition, this paper shows the existence of a unique solution to the distribution dependent stochastic differential delay equation.


\end{abstract}

\section{Introduction}

In this paper we consider the distribution dependent stochastic differential delay equation (SDDE)
\begin{equation}\label{MV-Delay-SDE}
    dX(t)=b(X(t),X(t-\tau),\bm{\mu}_t)dt+\sigma(X(t),X(t-\tau))dW(t)
\end{equation}
  where $b:\R^d \times \R^d \times \mathcal{P}(\R^d \times \R^d) \to \R^d$, $\sigma: \R^d \times \R^d  \to \R^{d \times d}, $ $\tau>0$,  $\bm{\mu}_t=\mathcal{L}(X(t),X(t-\tau))\in \mathcal{P}(\R^d \times \R^d)$ is the joint distribution of $X(t)$ and $X(t-\tau)$, and $W(t)$ is a $d$-dimensional Brownian motion on a complete probability space with natural filtration $(\Omega, \mathcal{F}, \P,\{\mathcal{F}_t\}_{t\geq 0})$.  We define for simplicity of notation $\mathbf{X}(t):=(X(t-\tau),X(t))$. Assume $X(t)=\xi(t)$ for $t\in[-\tau,0]$ where $ \xi \in C([-\tau,0];  \R^d)$ is an $\mathcal{F}_0$-measurable random variable.  We will show that under sufficient conditions on the coefficients $b$ and $\sigma$, \eqref{MV-Delay-SDE} admits a unique solution.

  Equation \eqref{MV-Delay-SDE}  is an
SDE with delay whose drift and diffusion coefficients
depend not only on the state process with delay, but also the joint distribution of the process with delay.  Without delay, this type of equation naturally arises as the limit of system of interacting ``particles.''  If the particles have symmetric dynamics, and the positions depend on the positions of the mean fields of the other particles, then as the number of particles tends to infinity, the limit dynamic of each particle does not depend on the positions of the others anymore but only on their statistical distributions by the law of large numbers. The resulting limit system is a McKean-Vlasov process described by the equation 
\[dX(t)=b(X(t),\mu_t)dt+\sigma(X(t))dW(t).\]

The study of this type of SDE was
initiated in \cite{McKean-66,McKean-67} and further developed in \cite{Funaki-84,Sznitman-91}.
Due to the introduction of mean-field games independently by \cite{HuangMC-06} and \cite{LasryL-07}, McKean-Vlasov SDEs have received growing attention, with more the recent developments found in \cite{DingQ-21,Chau-20,Wang-18} and the references therein.

We motivate the additional delay term with the following example:
\[dX(t)=b(X(t),\mu_t)dt+u(X(t-\tau),\mu_{t-\tau})+\sigma(X(t))dW(t).\]
The term $u$ can be thought of as a delayed control, of which there are many applications in science and engineering, see for example \cite{CarmoD-18I,CarmoD-18II}.  We refer to \cite{HairMS-11} and the references therein for a review of stochastic differential delay equations.  There is little work on distribution dependent SDDEs.  Some work has been done recently by \cite{WuXZ-23} on the functional case, and previously the case of discrete time feedback has been studied by \cite{WuHJG-22}. 

Another model of particular interest is the so called stochastic opinion dynamic model.  Such models attempt to measure peoples opinions of certain social scenarios.  Stochastic opinion dynamic models are often used to study marketing or other types of information campaigns, see for example \cite{Helbing-Dirk-10,Friedkin-Jojnsen-90}.  McKean-Vlasov SDEs turn out to be excellent models for such situations as each individuals opinion is influenced by the opinions of those they interact with.  Furthermore, information typically takes some amount of time to travel between people, so models including time delay for this information are also of interest, see for example \cite{Zhou-Zhao-Yuan-21}.

The purpose of this paper is to find appropriate Lyapunov and Lipschitz conditions to prove the well-posedness of \eqref{MV-Delay-SDE}.  The results in this paper expand upon the existence and uniqueness results for McKean-Vlasov SDEs found in \cite{Wang-23} in two ways.  First we introduce a measure dependence to the Lyapunov condition, and second we introduce the delay term $X(t-\tau)$ to both the drift and diffusion terms.  This drift term is accounted for in both our Lyapunov and Lipschitz conditions.  The results will then be expanded to show that in fact the problem with multiple non-constant deterministic delay terms is also well posed.

\subsection{Definitions}

Throughout this paper, we define $|\cdot|$ to be the Euclidean norm on $\R^d$, $\|\cdot\|$ to be the $L^1$ norm on $\mathbf{x}=(x_{-1},x_0)\in \R^d \times \R^d$ 
given by the average $\|\mathbf{x}\|:=\frac{1}{2}(|x_{-1}|+|x_0|)$ and we denote by $\|\cdot\|_F$ the Frobenius matrix norm.
For any   $\mathbf{x}=(x_{-1},x_0)\in \R^d \times \R^d$
and ${V}\in C^2(\R^{d};\R^{+})$, we write   $\mathbf{V}(\mathbf{x}):=V(x_0)+V(x_{-1})$.
Further define $\mathcal{P}_V (\R^d \times \R^d)$ to be the set of probability measures $\bm{\mu}$ such that $\int \mathbf{V}(\mathbf{x}) d\bm{\mu} < \infty$.
  We  next define, as in \cite{Wang-23} the special case of the Wasserstein distance by:
\[\mathcal{W}_{\psi, V}(\bm{\mu},\bm{\nu})=\inf_{\pi \in \Pi(\bm{\mu},\bm{\nu})}\int_{\R^{4d}} \psi(||\mathbf{x}-\mathbf{y}||)(1+ \mathbf{V}(\mathbf{x})+ \mathbf{V}(\mathbf{y}) )\pi(d\mathbf{x},d\mathbf{y}), \ \psi \in \Psi\]
where $\psi \in C^2(\R^+ ; \R^+)$ is an increasing function with $\psi(0)=0$ and $\psi''\geq 0$.
  Notice \eqref{MV-Delay-SDE} has generator 
 \[LV(\mathbf{x},\bm{\mu})=\partial_{x_0} V(x_0) b(\mathbf{x},\bm{\mu})+\frac{1}{2}\tr (({\sigma}(\mathbf{x},\bm{\mu})^\top \partial_{x_0x_0}V(x_0){\sigma}(\mathbf{x},\bm{\mu}))).\]
  Now we define $\bm{\mu} (\mathbf{V}) =\int_{\R^d \times \R^d}[V(x_0)+V(x_{-1}) ]\bm{\mu}(dx_0,dx_{-1})$.  Last we define well-posedness by:

    \begin{definition}
        A continuous adapted process $(X(t))_{t\geq 0}$ is called a strong solution to \eqref{MV-Delay-SDE} if
        \[\int_0^t \E[|b(\mbf{X}(s),\bm\mu(s)|+\|\sigma(\mbf{X}(s)\|_F^2]ds < \infty, \ \ t>0,\]
        and
        \[X(t)=X(0)+\int_0^t b(\mbf X(s),\bm \mu(s)) ds + \int_0^t \sigma(\mbf X(s))dW(s), \ \ t>0 \ \text{a.s.}\]
          Further, \eqref{MV-Delay-SDE} admits a unique strong solution in $\mathcal{P}_f$ if for any $\mathcal{F}_0$-measurable random variable $X(0)$, with $\E[f(X(0))] < \infty$, then \eqref{MV-Delay-SDE} has a unique solution $(X(t))_{t\geq 0}$ with $\E[f(X(t))] < \infty$.
    \end{definition}

\section{Well-Posedness}\label{Section Lyapunov Contraction}

We will assume the following throughout:
\begin{assumption}[Lyapunov Condition]\label{Lyapunov-assump-2}
    There exist constants $c_1$ and $c_2$ and a function $V\in C^2(\R^{d}:\R^+)$ such that $c_1|x|^p\leq V(x) \leq c_2|x|^p$ for some $p\geq 1$, and  $|\sigma(\mbf{x})\nabla {V}({x})|\leq \zeta(1+{V}({x}))$  with
\begin{equation}\label{Lyap-eq 2}
LV(\mathbf{x},\bm{\mu})\leq \zeta(1+\bm{\mu}(\mathbf{V})+\mathbf{V}(\mathbf{x})).
\end{equation}
\end{assumption}

\begin{assumption}[Generalized Lipshitz condition]\label{lip-assump}
    There exist constants  $K,\theta>0$, such that
    \begin{equation}\label{lip eq 1}
    \langle x_0-y_0, b(\mathbf{x},\bm{\mu})-b(\mathbf{y},\bm{\nu}) \rangle^+ + \frac{1}{2}||\sigma({\mbf x})-\sigma(\mbf y)||_F^2 \leq |x_{0}-y_{0}|(K\| \mathbf{x}-\mathbf{y} \|+\theta \W_{\psi,V}(\bm{\mu},\bm{\nu})),
    \end{equation} for all $\mbf x, \mbf y \in \R^{d}$ and $\bmu, \bnu \in \cP(\R^{d}\times \R^{d})$.
\end{assumption}

\begin{rem}
    We notice that the inclusion of the measure term $\bm \mu$ in assumption \ref{Lyapunov-assump-2}, generalizes the types of Lyapunov conditions often found in the literature which do not include any measure dependence, see for example \cite[Assumption $H_1$]{Wang-23}.  Further both assumptions \ref{Lyapunov-assump-2} and \ref{lip-assump} take into account the \textit{history} of the process.
\end{rem}

\subsection{Preliminary Results}

Now, for any $T>0$ and $\bm{\mu}\in C([0,T],\mathcal{P}(\R^d\times \R^d))$, consider the SDE
\begin{equation}\label{well-pose sde 1}
    dX^{\bm{\mu}}(t)=b(\mathbf{X}^{\bm{\mu}}(t),\bm{\mu}_t)dt+\sigma(\mathbf{X}^{\bm{\mu}}(t))dW(t), 
\end{equation}
with initial condition $X(t) = \xi(t) $ for all $t\in [-\tau, 0]$, 
where $\xi\in C([-\tau,0];\R^d) $ is  $\mathcal{F}_0$ measurable and satisfies 
   $\E[\|\xi \|_V]:= \E\big[\sup_{u\in [-\tau,0]} V(\xi 
   (u))\big] < \infty$.

\begin{lemma}\label{unique solution lem}
    Under assumptions  \ref{lip-assump} and \ref{Lyapunov-assump-2},    equation \eqref{well-pose sde 1} admits a unique solution.  Moreover the solution has the property
    \begin{equation}\label{unique solution eq 1}
        \E\bigg[\sup_{-\tau \leq s \leq T}V(X(t))<\infty\bigg]
    \end{equation}
\end{lemma}

\begin{proof}
    We notice that on the interval $t\in[0,\tau]$, we have
    \[dX^{\bm{\mu}}(t)=b(X^{\bm{\mu}}(t),\xi(t-\tau),\bm{\mu}_t)dt+\sigma(X^{\bm{\mu}}(t),\xi(t-\tau))dW(t),\]
    which is a classical SDE and the existence of a unique solution is well known with the property
    \[\E\bigg[\sup_{-\tau \leq s \leq \tau}V(X(t))<\infty\bigg].\]  See \cite{MeynT-93III}, for example. Once a solution on $[0,\tau]$ is known, we can write another classical SDE on $[\tau,2\tau]$ to obtain the unique solution on that interval.  Repeat this process for $[2\tau,3\tau]$, $[3\tau,4\tau]$ etc. until a solution is obtained on the entire interval $[0,T]$.
\end{proof}

Now for any $T>0$, let $\mu,\nu \in C([0,T];\mathcal{P}_V(\R^d \times \R^d)$, and consider two solutions to \eqref{well-pose sde 1} $X^{\bm{\mu}}(t)$ and $X^{\bm{\nu}}(t)$.  Define $Z^{\bm{\mu},\bm{\nu}}(t):=X^{\bm{\mu}}(t)-X^{\bm{\nu}}(t)$, $\bm{\beta}(t):=b(\mathbf{X}^{\bm{\mu}}(t),\bm{\mu}_t) - b(\mathbf{X}^{\bm{\nu}} (t),\bm{\nu}_t)$, and $\bm{\bm{\Delta}}(t):=\sigma(\mathbf{X}^{\bm{\mu}}(t))-\sigma(\mathbf{X}^{\bm{\nu}} (t))$.

\begin{lemma}\label{Z(t) lemma 2}
Given Assumption \ref{lip-assump}
    \begin{align}\label{Zt Process}
        |Z^{\bm{\mu},\bm{\nu}}(t)|\leq&\ |Z^{\bm{\mu},\bm{\nu}}(0)|+\int_0^t (K\|\mathbf{Z}^{\bm{\mu},\bm{\nu}}(u)\|+\theta\W_{\psi,V}(\bm{\mu}_u,\bm{\nu}_u))du \\
       \nonumber &+\int_0^t I_{\{|Z^{\bm{\mu},\bm{\nu}}(u)|\neq 0\}} \frac{1}{|Z^{\bm{\mu},\bm{\nu}}(u)|}\langle Z^{\bm{\mu},\bm{\nu}}(u), \bm{\bm{\Delta}}(u)dW(u) \rangle.
    \end{align}
\end{lemma}

\begin{proof}
    We first construct a function like that in \cite[proposition 5.2.13]{Karatzas-S}
    First notice that there exists a strictly decreasing sequence $\{a_n\} \subset (0,1]$ with $a_0=1$, $\lim_{n\to \infty}a_n=0$ and $\int_{a_n}^{a_{n-1}} \frac{1}{x}dx =n$ for every $n\geq 1$.  Then there exists a continuous function $\rho_n$ on $\R$ with support in $(a_n,a_{n-1})$ such that
 \begin{equation}
\label{e0gn(r)}
   0\leq \rho_n(x) \leq 2n^{-1}x^{-1}, \  \quad \text{and} \ \quad \int_{a_n}^{a_{n-1}}\rho_n(r)dr=1, \ \forall x > 0 \text{ and }n\in\mathbb{N}.
\end{equation}
    Consider the sequence of function
    \begin{equation}\label{gn(r)}
        g_n(r):=\int_{0}^{|r|}\int_0^y\rho_n(u)dudy, \quad r\in \R, \ n\geq1
    \end{equation}
    Notice $g_n$ is an even, twice differentiable function with $|g'_n(r)|\leq 1$, $\lim_{n\to \infty}g_n(r)=|r|$ for $r\in \R$ and $\lim_{n\to \infty}g'_n(r)=1$ for $r>0$.  We further notice $\{g_n(r)\}$ is non-decreasing.  Last we notice that $g_n,g'_n,$ and $g''_n$ vanish on the interval $(-a_n,a_n).$  Now direct computations yield 
    \[\nabla g_n (|x|)=g_{n}'(|x|)\frac{x}{|x|}, \quad \text{and} \quad \nabla^2g_n(|x|)=g''_n(|x|)\frac{xx^\top}{|x|^2}+g_{n}'(|x|)\bigg(\frac{I}{|x|}-\frac{xx^\top}{|x|^3}\bigg).\]
    Now since $dZ^{\bm{\mu},\bm{\nu}}(t)=\bm\beta(t)dt+\bm \Delta(t) dW(t)$ by It\^o's formula we have
    \begin{equation}\label{gn estimate}
    g_n(|Z^{\bm{\mu},\bm{\nu}}(t)|)=g_n(|Z^{\bm{\mu},\bm{\nu}}(0)|)+I_1+I_2+I_3, 
    \end{equation}
    where
    \begin{align*}
        I_1&:=\int_0^tI_{\{Z^{\bm{\mu},\bm{\nu}}(u)\neq 0\}}\frac{g'_n(|Z^{\bm{\mu},\bm{\nu}}(u)|)}{|Z^{\bm{\mu},\bm{\nu}}(u)|)}\bigg(\langle Z^{\bm{\mu},\bm{\nu}}(u), \bm\beta(t) \rangle + \frac{1}{2}\|\bm{\bm{\Delta}}(t)\|_F^2\bigg)du, \\    
        I_2&:= \int_{0}^tI_{\{Z^{\bm{\mu},\bm{\nu}}(u)\neq 0\}}\bigg(g''_n(|Z^{\bm{\mu},\bm{\nu}}(u)|)-\frac{g_{n}'(|Z^{\bm{\mu},\bm{\nu}}(u)|)}{|Z^{\bm{\mu},\bm{\nu}}(u)|}\bigg) \frac{|\bm{\bm{\Delta}}^\top (u)Z^{\bm{\mu},\bm{\nu}}(u)|^2}{|Z^{\bm{\mu},\bm{\nu}}(u)|^2} du, \\   
         I_3&:=\int_0^tI_{\{Z^{\bm{\mu},\bm{\nu}}(u)\neq 0\}}\frac{g'_{n}(|Z^{\bm{\mu},\bm{\nu}}(s)|)}{|Z^{\bm{\mu},\bm{\nu}}(u)|}\langle Z^{\bm{\mu},\bm{\nu}}(u),\bm{\bm{\Delta}}(u)dW(u) \rangle.
    \end{align*}
    Now by \eqref{lip eq 1} and the fact that $|g'_n(r)|\leq1$ we have
    \begin{align}\label{Zmu_lem_I1}
      \nonumber  I_1 \leq & \int_0^t I_{\{Z^{\bm{\mu},\bm{\nu}}(u)\neq 0\}}\frac{g'_n(|Z^{\bm{\mu},\bm{\nu}}(u)|)}{|Z^{\bm{\mu},\bm{\nu}}(s)|} |Z^{\bm{\mu},\bm{\nu}}(s)|(K \|\mathbf{Z}^{\bm{\mu},\bm{\nu}}(u)\|+\theta \W_{\psi,V}(\bm{\mu}_u,\bm{\nu}_u))du \\
        \leq & \int_0^t K(\|\mathbf{Z}^{\bm{\mu},\bm{\nu}}(u)\|+\theta \W_{\psi,V}(\bm{\mu}_u,\bm{\nu}_u))du.
    \end{align}
    Notice that direct computation yields
    \[g_n''(r)=\rho_n(r)\leq \frac{2}{nr}I_{(a_n,a_{n-1})}(r)\]
    and so by \eqref{lip eq 1} we have
    \begin{align}\label{Zmu_lem_I2}
    \nonumber    I_2 & \leq  \int_0^t I_{\{Z^{\bm{\mu},\bm{\nu}}(u) \neq 0\}}g''_n(|Z^{\bm{\mu},\bm{\nu}}(u)|) \frac{|\bm{\bm{\Delta}}^\top (u)Z^{\bm{\mu},\bm{\nu}}(u)|^2}{|Z^{\bm{\mu},\bm{\nu}}(u)|^2} du \\
     \nonumber  & \leq   \int_0^t I_{\{Z^{\bm{\mu},\bm{\nu}}(u)}\neq 0\} \frac{2I_{\{|Z^{\bm{\mu},\bm{\nu}}(u)| \in (a_n,a_{n-1})\}}}{n|Z^{\bm{\mu},\bm{\nu}}(u) |} \|\bm{\bm{\Delta}}(u)\|^2_F du\\
 \nonumber     &  \leq   \int_0^t I_{\{Z^{\bm{\mu},\bm{\nu}}(u)\in(a_n,a_{n-1})\}}  \frac{2K \|\mathbf{Z}^{\bm{\mu},\bm{\nu}}(u)\|}{n}  du \\&  \leq    \frac{2Kt \sup_{u\in[0,t]}\|\mathbf{Z}^{\bm{\mu},\bm{\nu}}(u)\|}{n} \to 0 \quad \text{as} \quad n \to \infty.
    \end{align}
    Last we notice that by \eqref{lip eq 1} we have
    \begin{align*}
        \int_0^t I_{\{Z^{\bm{\mu},\bm{\nu}}(u)\neq 0\}} \frac{g'_n(|Z^{\bm{\mu},\bm{\nu}}(u)|)^2}{|Z^{\bm{\mu},\bm{\nu}}(u)|^2}|\bm{\bm{\Delta}}(u)^\top Z^{\bm{\mu},\bm{\nu}}(u)|^2du  \leq \int_0^t \|\bm \Delta(u)\|^2du \leq \int_0^t K |Z^{\bm{\mu},\bm{\nu}}(u)|\|\mbf Z^{\bm{\mu},\bm{\nu}}(u)\|du. 
    \end{align*}
  We notice that
  \[  \E[|Z^{\bm{\mu},\bm{\nu}}(u)|\|\mbf Z^{\bm{\mu},\bm{\nu}}(u)\|]\leq \E\bigg[\sup_{-\tau \leq u \leq T}|X(u)|^2\bigg]\leq \E\bigg[\sup_{-\tau \leq u \leq T}V(X(t))\bigg]<\infty\]
  by equation \eqref{unique solution eq 1}.  Thus by the dominated convergence theorem,  we have
    \begin{align*}
        \E\bigg[\int_0^t I_{\{Z^{\bm{\mu},\bm{\nu}}(u)\neq 0\}}\frac{g'_n(|Z^{\bm{\mu},\bm{\nu}}(u)|)^2}{|Z^{\bm{\mu},\bm{\nu}}(u)|^2}|\bm{\bm{\Delta}}(u)^\top Z^{\bm{\mu},\bm{\nu}}(u)|^2du\bigg] \\ \to \E\bigg[\int_0^t I_{\{Z^{\bm{\mu},\bm{\nu}}(u) \neq 0\}}\frac{|\bm{\bm{\Delta}}(u)^\top Z^{\bm{\mu},\bm{\nu}}(u)|^2}{|Z^{\bm{\mu},\bm{\nu}}(u)|^2}du\bigg].  
    \end{align*}
    as $n \to \infty$.  Therefore upon a subsequence $\{n_k\}$
    \begin{equation}\label{Zmu_lem_I3}
    I_3\to \int_0^t I_{\{Z^{\bm{\mu},\bm{\nu}}(u)\neq 0\}}\frac{1}{|Z^{\bm{\mu},\bm{\nu}}(u)|}\langle Z^{\bm{\mu},\bm{\nu}}(u), \bm{\bm{\Delta}}(u)dW(u)\rangle \ \text{ a.s.}
    \end{equation}
    as $n_k\to \infty$.  Therefore  by plugging \eqref{Zmu_lem_I1}, \eqref{Zmu_lem_I2}, and \eqref{Zmu_lem_I3} into \eqref{gn estimate} and recalling that $\lim_{n\to \infty} g_n(r)=|r|$, we obtain \eqref{Zt Process}.  
\end{proof}

\subsection{Main Result}

We will now use a fixed point argument to show that equation \eqref{MV-Delay-SDE} has a unique solution.

\begin{theorem}\label{unique solution}
    Suppose Assumption \ref{Lyapunov-assump-2} holds and that \eqref{lip eq 1} holds, then equation \eqref{MV-Delay-SDE} admits a unique strong solution with
    \begin{equation}\label{EV(X) bound}
        \E[V(X(t))]\leq \bigg(\E[V(X(0))]+  \zeta T + 2\zeta \E[\|\xi\|_V]\bigg)e^{4\zeta T}<\infty.       
    \end{equation}

\end{theorem}

    \begin{proof}
    We first notice that \eqref{EV(X) bound} follows from It\^o's formula and Gronwall's inequality.  Indeed by Assumption \ref{Lyapunov-assump-2} and the fact that $V(x)\geq 0$ we have 
    \begin{align*}
        V(X(t)) & =V(X(0)+\int_0^t LV(X(u),\bmu_u)du +M(t) \\
        &\leq V(X(0))+\int_0^t \zeta [1+ \bm{\mu}_u(\mathbf{V})+V(X(u))+V(X(u-\tau))]du + M(t) \\
        &\leq V(X(0))+\int_0^t\zeta(1+\bm{\mu}_u(\mathbf{V}))du + \zeta\int_{-\tau}^0 V(X(u))du  \\ &\quad  +\zeta\int_0^{t-\tau}V(X(u))du + \zeta \int_0^t V(X(u))du +M(t) \\
        & \leq V(X(0))+\int_0^t\zeta(1+\bm{\mu}_u(\mathbf{V}))du + \zeta\int_{-\tau}^0 V(\xi(u))du+2\zeta  \int_0^t V(X(u))du + M(t)
    \end{align*}
   for a martingale term $M(t)$.   Taking expectation and recalling that $\bm\mu_{u}(\mbf V) = \E[V(X(u)) + V(X(u-\tau))]$, we get
    \begin{align*}
        \E[V(X(t))]\leq \E[V(\xi(0))]+\zeta t + 2\zeta\int_{-\tau}^0 \E[V(\xi(u))]du +4\zeta  \int_0^t \E[V(X(u))]du
    \end{align*}
    and by Gronwall's inequality we get
    \begin{align*}
        \E[V(X(t))]&\leq \bigg(\E[V(X(0))]+ \zeta t+\E\bigg[\int_{-\tau}^0 2\zeta V(\xi(u))du]\bigg)e^{4\zeta t} \\
        &\leq \bigg(\E[V(X(0))]+  \zeta T + 2\tau\zeta \E[\|\xi\|_V]\bigg)e^{4\zeta T}<\infty. 
    \end{align*}
    This establishes \eqref{EV(X) bound}.  Now we consider the SDE \eqref{well-pose sde 1} and notice
    \[\sup_{t \in [0,T]}\E[V(X(t))]<\infty.\]
    Next we will adopt a contraction argument similar to that in the proof of \cite[Lemma 2.3]{Wang-23}.
    For fixed $T>0$ we define 
    \[\mathcal{P}_{V,T}=\{\bm{\mu} \in C([0,T];\mathcal{P}_V (\R^d \times \R^d)): \bm{\mu}_0=\mathcal{L}(\mathbf{X}(0))\}.\]
     For $0 \leq t \leq T$, define $(H(\bm{\mu}))_t:=\mathcal{L}(\mathbf{X}^{\bm{\mu}}(t))$.  Notice that $\mathbf{X}^{\bm{\mu}}(t)$ is continuous and so $(H(\bm{\mu}))_t\in C([0,T];\mathcal{P}_V (\R^d \times \R^d))$.  Next we define
     \begin{equation*} 
     \mathcal{P}_{V,T}^N:= \bigg\{\bm \mu \in \mathcal{P}_{V,T}:   \sup_{t \in [0,T]} e^{-N t}\bm{\mu}_t(\mathbf{V})\leq N\big(1+\bm{\mu}_0(\mathbf{V})+\E\big[\|\xi \|_V]\big)\bigg \}. 
    \end{equation*}
    Notice that as $N \to \infty$, $\mathcal{P}_{V,T}^N \to \mathcal{P}_{V,T}$.  Therefore we will show that for any $N$ large enough, $H$ is a contraction in $\mathcal{P}_{V,T}^N$
    under
    \[ \W_{\psi,V,\lambda}(\bm{\mu},\bm{\nu}):=\sup_{0\leq t \leq T} [e^{-\lambda t}\W_{\psi,V}(\bm{\mu}_t,\bm{\nu}_t)], \quad \forall \bm{\mu}, \bm{\nu} \in \mathcal{P}_{V,T}^N\]
    for $\lambda >0.$
    To this end, we will show the following
    \begin{enumerate}
        \item $H:\mathcal{P}_{V,T}^N \to \mathcal{P}_{V,T}^N$
        \item There exists a $\lambda>0$ such that for any $\bm\mu, \bm\nu \in   \mathcal{P}_{V,T}^N$, we have \begin{equation}\label{contraction result}
        \W_{\psi,V,\lambda}(H(\bm{\mu}),H(\bm{\nu}))\leq \frac{1}{2}\W_{\psi,V,\lambda}(\bm{\mu},\bm{\nu}).
        \end{equation}
    \end{enumerate}
 
    Now by It\^o's formula and assumption \ref{Lyapunov-assump-2} we have for $s=0,\tau$ with $t>0$ when $s=\tau$
    \begin{align*}
        \E[V&(X^{\bm{\mu}}(t-s))]\leq \bm{\mu}_0(\mathbf{V}) + \int_0^{t-s} \zeta(1+\bm{\mu}_u(\mathbf{V}) +\E[V(X^{\bm{\mu}}(u))]+\E[V(X^{\bm{\mu}}(u-\tau))])du \\
        & \leq \bm{\mu}_0(\mathbf{V}) + \int_0^t \zeta(1+\bm{\mu}_u(\mathbf{V}) +\E[V(X^{\bm{\mu}}(u))]+\E[V(X^{\bm{\mu}}(u-\tau))])du \\
        & \leq \bm{\mu}_0(\mathbf{V}) + \int_0^t \zeta(1+\bm{\mu}_u(\mathbf{V})e^{-Nu})e^{Nu} du+ \zeta \int_{-\tau}^0 \E[V(X^{\bm{\mu}}(u))]du \\
        & \quad+ \int_0^t 2\zeta \E[V(X^{\bm{\mu}}(u))]du \\
        & \leq \bm{\mu}_0(\mathbf{V}) + \int_0^t \zeta\bigg(1+N\bigg(1+\bm{\mu}_0(\mathbf{V})+\E\big[\|\xi \|_V]\bigg)\bigg)e^{Nu}du + \tau\zeta \E\big[\|\xi \|_V] \\
        &\quad + \int_0^t 2\zeta \E[V(X^{\bm{\mu}}(u))]du.
    \end{align*}
    In the case where $s=\tau$ and $t\in [-\tau,0]$ we have clearly \[\E[V(X^{\bm{\mu}}(t-\tau))]\leq \E\big[\|\xi \|_V] \leq \max\{1,\tau \zeta\} \E\big[\|\xi \|_V].\] By Gronwall's inequality and multiplying by $e^{-Nt}$ we get 
    \begin{align*}
        &e^{-Nt}\E[V(X^{\bm{\mu}}(t-s))] \\
        &\leq \bigg( \bm{\mu}_0(\mathbf{V}) + \bar{\zeta} \E\big[\|\xi \|_V] + \int_0^t  \zeta (1+N(1+\bm{\mu}_0(\mathbf{V})+\E\big[\|\xi \|_V]))e^{Nu}du\bigg)e^{(2\zeta-N) t} \\
        &\leq \bigg(\bm{\mu}_0(\mathbf{V}) + \bar{\zeta} \E\big[\|\xi \|_V] \bigg)e^{(2\zeta-N)t} +\zeta e^{2\zeta t}\bigg(1+N\bigg(1+\bm{\mu}_0(\mathbf{V})+\E\big[\|\xi \|_V]\bigg)\bigg)\int_0^t e^{N(u-t)}du \\
        &\leq c\bigg(\bm{\mu}_0(\mathbf{V})+\bar{\zeta}\E\big[\|\xi \|_V] \bigg)e^{-Nt}+\zeta c \bigg(1+\bigg(1+\bm{\mu}_0(\mathbf{V})+\E\big[\|\xi \|_V]\bigg)\bigg) \\
        &\leq \bar{\zeta}c\bigg(1+\bm{\mu}_0(\mathbf{V})+\E\big[\|\xi \|_V]\bigg)+\bar{\zeta}\bigg(2+2\bm{\mu}_0(\mathbf{V})+2\E\big[\|\xi \|_V]\bigg) \\
        &\leq 3 \bar{\zeta}c\bigg(1+\bm{\mu}_0(\mathbf{V})+\E\big[\|\xi \|_V]\bigg) \leq \frac{1}{2}N_0\bigg(1+\bm{\mu}_0(\mathbf{V})+\E\big[\|\xi \|_V]\bigg),
    \end{align*}
    where
    \[c:=e^{2\zeta T}, \ \bar{\zeta}:=\max \{1,\zeta,\tau\zeta\}, \ N_0:=6\bar{\zeta}c.\]
    Summing for $s=0,\tau$ gives us
    \[e^{-Nt}\E[V(X^{\bm{\mu}}(t)) +V(X^{\bm{\mu}}(t-\tau)) ] \le N\big(1+\bm{\mu}_0(\mathbf{V})+\E[\|\xi\|_V]\big)\]
    for all $N\geq N_0$.  Thus we have  $H:\mathcal{P}_{V,T}^N \to \mathcal{P}_{V,T}^N$.
    
    Next we show there exists a $\lambda >0$ such that
    \[\W_{\psi,V, \lambda}(H(\bm{\mu}),H(\bm{\nu}))\leq \frac{1}{2}\W_{\psi,V \lambda}(\bm{\mu},\bm{\nu}).\]
    By Lemma \ref{Z(t) lemma 2} and by the fact that $X^{\bm{\mu}}(t)=X^{\bm{\nu}}(t)=\xi(t)$ for $t\in[-\tau,0]$ we have for $s=0,\tau$
    \begin{align*}
        |Z^{\bm{\mu},\bm{\nu}}(t-s)|\leq&\int_0^{t-s} K(\|\mathbf{Z}^{\bm{\mu},\bm{\nu}}(u)\|+\theta\W_{\psi,V}(\bm{\mu}_u,\bm{\nu}_u))du \\
       \nonumber &+\int_0^{t-s} I_{\{|Z^{\bm{\mu},\bm{\nu}}(u)|\neq 0\}} \frac{1}{|Z^{\bm{\mu},\bm{\nu}}(u)|}\langle Z^{\bm{\mu},\bm{\nu}}(u), \bm{\bm{\Delta}}(u)dW(u) \rangle.
    \end{align*}
   For the case $t\in[-\tau,0]$ and $s=\tau$ we simply have $|Z^{\bm{\mu},\bm{\nu}}(t-\tau)|=0$. Then adding for $s=0,\tau$ gives,
    \begin{align}\label{Wellposed SDE 1}
\nonumber \frac{1}{2}(|Z^{\bm{\mu},\bm{\nu}}(t)| + & |Z^{\bm{\mu},\bm{\nu}}(t-\tau)|)=\|\mathbf{Z}^{\bm{\mu},\bm{\nu}}(t)\| \\
\nonumber \leq &\  \int_0^t K(\|\mathbf{Z}^{\bm{\mu},\bm{\nu}}(u)\|+\theta\W_{\psi,V}(\bm{\mu}_u,\bm{\nu}_u))du \\
        &+\frac{1}{2}\int_0^t I_{\{|Z^{\bm{\mu},\bm{\nu}}(u)|\neq 0\}} \frac{1+I_{\{u\in[0,t-\tau]\}}}{|Z^{\bm{\mu},\bm{\nu}}(u)|}\langle Z^{\bm{\mu},\bm{\nu}}(u), \bm{\bm{\Delta}}(u) dW(u) \rangle.
    \end{align}
    Now by \cite{Wang-23}[Lemma 2.2]
    \begin{align}\label{well-posed psi}
 \nonumber       \psi(\|\mathbf{Z}^{\bm{\mu},\bm{\nu}}(t)\|)
        \leq&  \int_0^t \psi'(\|\mathbf{Z}^{\bm{\mu},\bm{\nu}}(u)\|)(K\|\mathbf{Z}^{\bm{\mu},\bm{\nu}}(u)\|+\theta \W_{\psi,V}(\bm{\mu}_u,\bm{\nu}_u))d u \\
        &+\frac{1}{8} \int_0^t I_{\{|Z^{\bm{\mu},\bm{\nu}}(u)|\neq 0\}} \psi''(\|\mathbf{Z}^{\bm{\mu},\bm{\nu}}(u)\|)\big(1+I_{\{u\in[0,t-\tau]\}}\big)^2(||\bm{\bm{\Delta}}(u)||_F^2)du  \\
 \nonumber       &+\frac{1}{2}\int_0^t I_{\{|Z^{\bm{\mu},\bm{\nu}}(u)|\neq 0\}} \psi'(\|\mathbf{Z}^{\bm{\mu},\bm{\nu}}(u)\|) \frac{1+I_{\{u\in[0,t-\tau]\}}}{|Z^{\bm{\mu},\bm{\nu}}(u)|}\langle Z^{\bm{\mu},\bm{\nu}}(u), \bm{\bm{\Delta}}(u) dW(u) \rangle \\
  \nonumber      \leq&  \int_0^t C_0(\psi(\|\mathbf{Z}^{\bm{\mu},\bm{\nu}}(u)\|)+\W_{\psi,V}(\bm{\mu}_u,\bm{\nu}_u)) d u \\
   \nonumber       &+\frac{1}{2}\int_0^t I_{\{|Z^{\bm{\mu},\bm{\nu}}(u)|\neq 0\}} \psi'(\|\mathbf{Z}^{\bm{\mu},\bm{\nu}}(u)\|) \frac{1+I_{\{u\in[0,t-\tau]\}}}{|Z^{\bm{\mu},\bm{\nu}}(u)|}\langle Z^{\bm{\mu},\bm{\nu}}(u), \bm{\bm{\Delta}}(u) dW(u) \rangle,
    \end{align} 
    for some constant $C_0$.  The last line follows since $\psi$ is a concave function.  
    Next we notice for $s=0,\tau$ with $t>0$ when $s=t$ we have by It\^o's formula 
    \begin{align*}
        V&(X^{\bm{\mu}}(t-s))  = \|\xi\|_v+\int_0^{t-s} LV(\mathbf{X}^{\bm{\mu}}(u),\bm{\mu}_u)du +\int_0^{t-s}\langle\sigma(X^{\bm{\mu}}(u))\nabla V(X^{\bm{\mu}}(u)),dW(u)\rangle \\
        &\leq \|\xi\|_v+\int_0^{t-s} \zeta(1+\bm{\mu}_u(\mathbf{V})+\mathbf{V}(\mathbf{X}^{\bm{\mu}}(u)))du +\int_0^{t-s}\langle\sigma(X^{\bm{\mu}}(u))\nabla V(X^{\bm{\mu}}(u)),dW(u)\rangle \\
        &\leq \|\xi\|_v+\int_0^{t} \zeta(1+\bm{\mu}_u(\mathbf{V})+\mathbf{V}(\mathbf{X}^{\bm{\mu}}(u)))du 
         +\int_0^{t}I_{u\in [0,t-s]}\langle\sigma(X^{\bm{\mu}}(u))\nabla V(X^{\bm{\mu}}(u)),dW(u)\rangle.
    \end{align*}
    Again, for the case $t\in[-\tau,0]$ and $s=\tau$ the last inequality still holds as $V(X^{\bm{\mu}}(t-s))\leq \|\xi\|_V$, and the remainder of the terms are positive. We have a similar estimate for $V(X^{\bm{\nu}}(t-s))$.  Thus
    \begin{align}\label{well posed V}
\nonumber        \mathbf{V}(\mathbf{X}&^{\bm{\mu}}(t))+\mathbf{V}(\mathbf{X}^{\bm{\nu}}(t)) \\
     \nonumber   \leq & 2\|\xi\|_v +2\int_0^{t} \zeta(2+\bm{\mu}_u(\mathbf{V})+\bm{\nu}_u(\mathbf{V})+\mathbf{V}(\mathbf{X}^{\bm{\mu}}(u))+\mathbf{V}(\mathbf{X}^{\bm{\nu}}(u)))du \\
\nonumber     &+ \int_0^{t}(1+I_{u\in [0,t-s]})\langle\sigma(X^{\bm{\mu}}(u))\nabla V(X^{\bm{\mu}}(u)+\sigma(X^{\bm{\nu}}(u))\nabla V(X^{\bm{\nu}}(u)),dW(u)\rangle \\
 \nonumber       \leq & 2\|\xi\|_v + C_1(N)+\int_0^{t} 2\zeta(1+\mathbf{V}(\mathbf{X}^{\bm{\mu}}(u)+\mbf V(\mathbf{X}^{\bm{\nu}}(u))))du \\
             &+ \int_0^{t}(1+I_{u\in [0,t-s]})\langle\sigma(X^{\bm{\mu}}(u))\nabla V(X^{\bm{\mu}}(u)+\sigma(X^{\bm{\nu}}(u))\nabla V(X^{\bm{\nu}}(u)),dW(u)\rangle
    \end{align}
    for some constant $C_1(N)$ depending on $N$.  The last line follows by recalling that $\bm{\mu}_t(\mathbf{V}) \leq N(1+\bm{\mu}_0(\mathbf{V})+\E[\|\xi\|_V])e^{NT}$.  Next we calculate the cross variation
    \begin{align}\label{Cross-var-bound}
 \nonumber   [\psi&(\|\mathbf{Z}^{\bm{\mu},\bm{\nu}}\|),1+\mathbf{V}(\mathbf{X}^{\bm{\mu}})+V(\mathbf{X}^{\bm{\nu}})]_t \\
 \nonumber &\leq 2\psi'(\|\mathbf{Z}^{\bm{\mu},\bm{\nu}}(t)\|)\langle \bm\Delta(t)e(t), \sigma(X^{\bm{\mu}}(t))\nabla V(X^{\bm{\mu}}(t)+\sigma(X^{\bm{\nu}}(t))\nabla V(X^{\bm{\nu}}(t)) \rangle \\
\nonumber    & \leq 2\psi'(\|\mathbf{Z}^{\bm{\mu},\bm{\nu}}(t)\|)\|\bm\Delta(t)\| |\sigma(X^{\bm{\mu}}(t))\nabla V(X^{\bm{\mu}}(t)+\sigma(X^{\bm{\nu}}(t))\nabla V(X^{\bm{\nu}}(t))| \\
    & \leq C_2\psi(\|\mathbf{Z}^{\bm{\mu},\bm{\nu}}(t)\|)(1+\mathbf{V}(\mathbf{X}^{\bm{\mu}}(t))+V(\mathbf{X}^{\bm{\nu}}(t))        
    \end{align}
    For some constant $C_2$.  The last line follows by assumptions \ref{Lyapunov-assump-2}, \ref{lip-assump}, and the fact that $\psi$ is a concave function.
    Now we define $$\rho(\|\mathbf{Z}^{\bm{\mu},\bm{\nu}}(t)\|):=\psi(\|\mathbf{Z}^{\bm{\mu},\bm{\nu}}(t)\|)(1+\mathbf{V}(\mathbf{X}^{\bm{\mu}}(t))+V(\mathbf{X}^{\bm{\nu}}(t))).$$ 
    Then combining \eqref{well-posed psi}, \eqref{well posed V}, and \eqref{Cross-var-bound} we get
    \begin{align*}
        d\rho(\|\mathbf{Z}^{\bm{\mu},\bm{\nu}}(t)\|) &=d\psi(\|\mathbf{Z}^{\bm{\mu},\bm{\nu}}(t)\|) (1+\mathbf{V}(\mathbf{X}^{\bm{\mu}}(t)+V(\mathbf{X}^{\bm{\nu}}(t)))+\psi(\|\mathbf{Z}^{\bm{\mu},\bm{\nu}}(t)\|)(d\mathbf{V}(\mathbf{X}^{\bm{\mu}}(t)+dV(\mathbf{X}^{\bm{\nu}}(t)))\\
        & \quad+d[\psi(\|\mathbf{Z}^{\bm{\mu},\bm{\nu}}\|),1+\mathbf{V}(\mathbf{X}^{\bm{\mu}})+V(\mathbf{X}^{\bm{\nu}})]_t\\
       & \leq   C_0(\psi(\|\mathbf{Z}^{\bm{\mu},\bm{\nu}}(t)\|) + \W_{\psi,V}(\bm{\mu}_t,\bm{\nu}_t))(1+\mathbf{V}(\mathbf{X}^{\bm{\mu}}(t))+V(\mathbf{X}^{\bm{\nu}}(t)))dt \\
        &\quad + \psi(\|\mathbf{Z}^{\bm{\mu},\bm{\nu}}(t)\|)(2 \zeta+C_2)(1+\mathbf{V}(\mathbf{X}^{\bm{\mu}}(t)+V(\mathbf{X}^{\bm{\nu}}(t))) +dM(t)\\
      &  \leq   C_3(N)\big(\rho(\|\mathbf{Z}^{\bm{\mu},\bm{\nu}}(t)\|)+ \W_{\psi,V}(\bm{\mu}_t,\bm{\nu}_t)(1+\mathbf{V}(\mathbf{X}^{\bm{\mu}}(t)+V(\mathbf{X}^{\bm{\nu}}(t))\big)+dM(t)
    \end{align*}
    for a constant $C_3(N)$ depending on $N$.  Now we notice that $\E[\mathbf V(\mathbf X^{\bm{\mu}}(t)]+\E[\mathbf{V}(\mathbf{X}^{\bm{\nu}}(t)]\leq N(2+\bm{\mu}_0(\mathbf{V})+\bm{\nu}_0(\mathbf{V})+2\E[\|\xi\|_V])e^{NT}:=D(N)$ and so by taking expectation we get
    \begin{align*}
        \E[\rho(\mathbf{Z}^{\bm{\mu},\bm{\nu}}(t))]&\leq C_3(N)\int_0^t \E[\rho(\mathbf{Z}^{\bm{\mu},\bm{\nu}}(u))]+ \W_{\psi,V}(\bm{\mu}_u,\bm{\nu}_u)(1+\E[\mathbf{V}(\mathbf{X}^{\bm{\mu}}(u)+\E[V(\mathbf{X}^{\bm{\nu}}(u)])du \\
        &\leq C_3(N)\int_0^t \E[\rho(\mathbf{Z}^{\bm{\mu},\bm{\nu}}(u))]du+C_3(N)(1+D(N))\int_0^t \W_{\psi,V}(\bm{\mu}_u,\bm{\nu}_u) du.
    \end{align*}
    Then by Gronwall's inequality and defining $C_4(N):=C_3(N)(1+D(N))$ we get
    \[\E[\rho(\mathbf{Z}^{\bm{\mu},\bm{\nu}}(u))]\leq C_4(N)\int_0^t \W_{\psi,V}(\bm{\mu}_u,\bm{\nu}_u) du e^{C_3(N)T}.\]
    Then multiplying by $e^{-\lambda t}$ gives
    \begin{align*}
    e^{-\lambda t}&\W_{\psi,V}(H(\bm \mu),H(\bm\nu))\leq e^{-\lambda t}\E[\rho(\mathbf{Z}^{\bm{\mu},\bm{\nu}}(t))] \\
    &\leq C_4(N)e^{C_3(N)T}\int_0^t e^{-\lambda(t-u)}e^{-\lambda u} \W_{\psi,V}(\bm{\mu}_u,\bm{\nu}_u) du  
    \leq \frac{1}{\lambda}C_4(N)e^{C_3(N)T}\W_{\psi,V,\lambda}(\bm\mu,\bm\nu).            
    \end{align*}
    Taking supremum over $t\in[0,T]$ and choosing $\lambda=2C_4(N)e^{C_3(N)T} $ gives the desired contraction, \eqref{contraction result}.  Therefore $H:\mathcal{P}_{V,T}^N \to \mathcal{P}_{V,T}^N$ is a contraction under $\W_{\psi,V,\lambda}$ for large enough $N$ and $\lambda$.  Thus $H$ has a unique fixed point $\bm{\mu}^*$, such that $X(t)^{\bm{\mu}^*}$  is the unique solution to \eqref{well-pose sde 1}, establishing the existence of a unique strong solution.
\end{proof}


\begin{example}
  We consider the following stochastic opinion dynamic model:
  \begin{equation}\label{SODM-ex-1}
  \begin{aligned}
      dX(t)=&\bigg[\int_\R\Phi_{-1}(|X(t-\tau)-y_{-1}|)(X(t-\tau)-y_{-1})\mu_{-1}(dy_{-1}) + \varphi_{-1}(X(t-\tau))\\
      &+\int_\R \Phi_0(|X(t)-y_0|)(X(t)-y_0)\mu_0(dy_0)      
      +\varphi_{0}(X(t))\bigg]dt+\sigma dW(t),
      \end{aligned}
  \end{equation}
with $X(t)=\xi(t)$ for $t\in[-\tau,0]$. Which is motivated by the Heterophilious dynamic model in \cite{Motsch-Sebastien-Tadmor-14}, and which is further studied numerically in \cite[section 4.2]{Lu-Maggioni-Tang-21}.  The above model is modified to include some \textit{intrinsic memory}.

Here $X(t)$ is an $\R$ valued random variable representing an individuals opinion.  Further $\mu_0$ represents the law of $X(t)$ and $\mu_{-1}$ represents the law of $X(t-\tau)$.  The functions $\Phi_0$ and $\Phi_{-1}$ are interaction kernels representing the influence of one person's opinion on another person's opinion.  In the case of Heterophilious dynamic models proposed in \cite{Motsch-Sebastien-Tadmor-14}, where one is more influenced by people with more differing opinions, but not extremely so, these functions are Lipschitz and compactly supported.  Finally, $\varphi_0$ and $\varphi_{-1}$ represent a drift in opinion, typically a reversion to $0$.

We assume $\Phi_{-1}$ and $\Phi_{0}$ are $a$-Lipschitz with compact support (and therefore bounded), and $\varphi_{-1}$ and $\varphi_0$ are $b$-Lipschitz.  Then we have for $V(x)=x^2$, $\|\Phi\|=\max\{\sup \Phi_{-1},\sup \Phi_0\}$,
\begin{align*}
    LV(\mbf{x},\bm{\mu})\leq&2x_0\bigg[\|\Phi\|\int_\R |x_{-1}|+|y_{-1}|\mu_{-1}(dy_{-1})+\|\Phi\|\int_\R|x_0|+|y_0|\mu_0(dy_0) \\
    &+|\varphi_0(0)|+|\varphi_{-1}(0)|+b(|x_0|+|x_{-1}|)\bigg]+\sigma^2\\
    \leq&6x_0^2+\frac{\|\Phi\|^2}{2}\bigg[x_{-1}^2+x_0^2+\int_\R\int_\R|y_{-1}|+|y_0|\mu(dy_{0},dy_{-1})\bigg]\\
    &+|\varphi_0(0)|^2+2b^2x_0^2+|\varphi_{-1}(0)|^2+2b^2x_{-1}^2+\sigma^2\\
    \leq& \zeta(1+\mbf{V}(\mbf{x})+\bm{\mu}(\mbf{V})),
\end{align*}
 for $\zeta=\max\{6+\frac{\|\Phi\|}{2}+2b^2,|\varphi_0(0)|^2+|\varphi_{-1}(0)|^2+\sigma^2\}$.  Therefore assumption \ref{lip-assump} holds.
Further notice
\begin{align*}
    \int_\R \int_\R & \Phi_0(|x_0-y_0|)(x_0-y_0) - \Phi_0(|\tilde{x}_0-\tilde{y}_0|)(\tilde{x}_0-\tilde{y}_0)\mu_0(dy_0)\tilde{\mu}_0(d\tilde{y}_0) \\
    \leq& \int_\R \int_R |\Phi_0(|x_0-y_0|)(x_0-y_0)-\Phi_0(|x_0-y_0|)(\tilde{x}_0-\tilde{y}_0)| \\
    &+ |\Phi_0(|x_0-y_0|)(\tilde{x}_0-\tilde{y}_0)-\Phi_0(|\tilde{x}_0-\tilde{y}_0|)(\tilde{x}_0-\tilde{y}_0)|\mu_0(dy_0)\tilde{\mu}_0(d\tilde{y}_0) \\
    \leq &\int_\R \int_\R \|\Phi\||(x_0-y_0)-(\tilde{x}_0-\tilde{y}_0)|+R||x_0-y_0|-|\tilde{x}_0-\tilde{y}_0||\mu_0(dy_0)\tilde{\mu}_0(d\tilde{y}_0 \\
    \leq& \int_\R \int_\R \|\Phi\|||x_0-\tilde{x}_0|+|y_0-\tilde{y}_0||+aR||x_0-\tilde{x}_0|-|y_0-\tilde{y}_0||\mu_0(dy_0)\tilde{\mu}_0(d\tilde{y}_0\\
    \leq & A|x_0-\tilde{y}_0|+A\int_\R\int_\R |y_0-\tilde{x}_0|\pi_0(dy_0,d\tilde{y}_0)
\end{align*}
for a coupling $\pi_0$ and constant $A=2\max\{\|\Phi\|,aR\}$ such that $\Phi_0$ and $\Phi_{-1}$ have support contained in an interval of length $R$. A similar result holds for the delayed kernel term.  Combining this with the fact that $\varphi_0$ and $\varphi_1$ are $b$-Lipschitz we have
\begin{align*}
    ( x_0-\tilde{x}_0)&(b(\mbf{x},\bm{\mu})-b(\mbf{\tilde{x}},\bm{\tilde{\mu}})) \leq (x_0-\tilde{x}_0)(A(|x_0-\tilde{x}_0|+|x_{-1}-\tilde{x}_{-1}|)+b(|x_0-\tilde{x}_0|+|x_{-1}-\tilde{x}_{-1}|)\\
    &+A\bigg(\iint |y_{0}-\tilde{y}_{0}|\pi_0(dy_0,d\tilde{y}_0)+\iint |y_{-1}-\tilde{y}_{-1}|\pi_{-1}(dy_{-1},d\tilde{y}_{-1})\bigg) \\
    \leq& |x_0-\tilde{x}_0|(K\|\mbf{x}-\mbf{\tilde{x}}\|+\theta \W_1(\bm{\mu},\bm{\tilde{\mu}}))\leq |x_0-\tilde{x}_0|(K\|\mbf{x}-\mbf{\tilde{x}}\|+\theta \W_{\|\cdot\|,V}(\bm{\mu},\bm{\tilde{\mu}}),
\end{align*}
for $K=2\max\{A,b\}$ and $\theta=A$.  The last inequality follows since $\W_1(\bm{\mu},\bm{\nu})\leq \W_{\|\cdot \|,V}(\bm{\mu},\bm{\nu}))$ for any $V$.  Therefore assumption \ref{lip-assump} holds, and by theorem \ref{unique solution}, \eqref{SODM-ex-1} admits a unique solution.
\end{example}

The above results can be expanded to the following distribution dependent SDE with $n$ deterministic delay terms

\begin{equation}\label{MVSDE_Det_Delay}
    dX(t)=b(\mbf X(t), \bm{\mu}_t)dt+\sigma (\mbf X(t))dW(t)
\end{equation}
where $\mbf X(t)=(X(t-\tau_n(t)),...X(t-\tau_1(t)),X(t))$ for $\tau_i(t): \R^+ \to (0,\tau]$ and $\bm \mu_t$ is the joint distribution of $\mbf X(t)$.  We still assume the initial condition $X(t)=\xi(t)$ for $t\in[-\tau,0]$.  We appropriately scale the $L^1$ norm $\|\cdot \|$ on $\mbf x=(x_{-n},...,x_{-1},x_0)$ by $\|\mbf x\|=\frac{1}{n}\sum_{i=0}^n |x_{-i}|$.  This scaling is to simplify the calculation of the radial process defined similarly to equation \eqref{Zt Process}. Further redefine $\mbf V(\mbf x):=\sum_{i=0}^n V(x_{-i})$ and $\bm{\mu}(\mbf V):=\int_{\R^{nd}} \sum_{i=0}^n V(x_{-i})\bm{\mu}(dx_0,...,dx_{-n})$. Then theorem \ref{unique solution} can be generalized by the following corollary.

\begin{corollary}
    Suppose Assumption \ref{Lyapunov-assump-2} and \ref{lip-assump} hold, then equation \eqref{MVSDE_Det_Delay} admits a unique strong solution with
           \[ \E[V(X(t))]\leq \bigg(\E[V(X(0))]+  \zeta T + 2(n-1)\zeta \E[\|\xi\|_V]\bigg)e^{2n \zeta T}<\infty.\]
\end{corollary}

The proof is omitted because it is identical to that of theorem \ref{unique solution} above, where we simply allow $s=0,\tau_1(t)...,\tau_i(t)$ where appropriate, and note that $\tau_i(t) \in (0,\tau]$.

\bibliographystyle{unsrtnat}
\bibliography{refs}

\end{document}